\documentclass[12pt]{amsart}
\usepackage{amscd}
\usepackage{verbatim}
\usepackage[russian,english]{babel}

\usepackage{amssymb, amsmath, amsthm, amscd,ifthen}
\usepackage[dvips]{graphics}
\usepackage[cp866]{inputenc}
\usepackage{graphicx}
\usepackage{epsfig}
\usepackage{hyperref}
\usepackage[all,cmtip]{xy}
\usepackage{tikz-cd}
\usepackage[all]{xy}

\usepackage{amsmath,amssymb,amscd,}
\usepackage[all]{xy}

\usepackage{color}

\usepackage[dvips]{graphics}
\usepackage[cp866]{inputenc}
\usepackage{graphicx}
\usepackage{epsfig}

\theoremstyle{plain}
\newtheorem{theorem}{Theorem}
\newtheorem{lemma}[theorem]{Lemma}
\newtheorem{corollary}[theorem]{Corollary}
\newtheorem{proposition}[theorem]{Proposition}
\newtheorem{claim}[theorem]{Claim}

\theoremstyle{definition}
\newtheorem{definition}{Definition}

\theoremstyle{plain}
\newtoks\thehProclaim
\newtheorem*{Proclaim}{\the\thehProclaim}

\theoremstyle{definition}
\newtoks{\thehRemark}
\newtheorem*{Remark}{\the\thehRemark}

\renewcommand{\leq}{\leqslant}
\renewcommand{\geq}{\geqslant}
\DeclareMathOperator{\Int}{Int}
\DeclareMathOperator{\ind}{ind}
\DeclareMathOperator{\Skel}{Skel}
\DeclareMathOperator{\sd}{sd}
\DeclareMathOperator{\Star}{Star}
\begin{document}

	\title{Affine transverse foliations in sphere bundles}

	\author{Ilya Alekseev, Ivan Nasonov,   Gaiane Panina }
	
	\address{I. Alekseev: St. Petersburg State University; G. Panina:  St. Petersburg State University; gaiane-panina@rambler.ru; Ivan Nasonov: St. Petersburg State University}

	

	\begin{abstract}  Let~$S^{n-1}\rightarrow E \rightarrow M^n$ be an oriented sphere bundle  supporting a {smooth} affine transverse foliation.

We  give a new and elementary proof of the following fact: if  the fundamental group  {$\pi_1(M^n)$} is amenable, then the Euler number of the bundle vanishes. 

As a by-product, we  give an upper bound for the Euler number of the bundle in the general case, that is, for non-amenable {$\pi_1(M^n)$}.

	\end{abstract}

	\maketitle

\section{Introduction}\label{SecIntro}

\subsection{Affinely foliated sphere bundles}
{An \textit{oriented sphere bundle} over a closed smooth oriented manifold~$M^n$ is }a locally trivial fiber bundle
$$\pi:E\rightarrow M^n$$ whose fibers are oriented spheres~$S^{n-1}$.

An \textit{oriented sphere bundle with a smooth transverse affine foliation}~$F$ (or a \textit{flat bundle}, for short) is an oriented sphere bundle that is obtained from the following construction:

There is an oriented real vector bundle~$\widetilde{E}\rightarrow M^n$ of rank~$n$.
The total space~$\widetilde{E}$ carries a \textit{transverse foliation}~$\widetilde{F}$ whose smooth~$n$-dimensional leaves are transverse to the fibers. 
Equivalently,~$\widetilde{E}\rightarrow M^n$ admits a flat connection.
The foliation~$\widetilde{F}$ is \textit{affine}: {the associated holonomies} are linear homomorphisms. 
In particular, the zero section is a leaf. 
The foliation~$\widetilde{F}$ induces a transverse foliation~$F$ on the associated sphere bundle~$E$.
     
\medskip

Once orientations of~$M^n$ and the fibers are fixed, the Euler class of the bundle may be identified with an integer~$\mathcal{E}$. 
When~$n$ is odd, the Euler number vanishes for any oriented sphere bundle, so {our results make sence in even dimensions only}. 


\medskip

\textbf{Example. } The trivial bundle~$M^n\times S^{n-1}\rightarrow M^n$ is flat since it has a trivial foliation by  leaves~$M^n\times \{\omega\}$.  Its Euler number is zero.

\bigskip

Our first aim is to give an elementary proof of the following:

\begin{theorem}\label{ThmMain}
Let~$M^n$ be a closed oriented smooth manifold  whose fundamental group is amenable. Then the Euler number of every flat sphere bundle  over $M^n$
 vanishes.
 
\end{theorem}


Our proof relies on ideas of D. Sullivan \cite{Sullivan} and  on M. Kazarian's averaging principle \cite{Kaz}, together with the F{\o}lner property of amenable groups.

\medskip

This technique gives also a bound for the Euler number in the general situation. Before we formulate our result, let us remind:


\begin{theorem}\label{ThmGromovSmillie} \cite{Gromov} (Milnor--Sullivan--Smillie Theorem)
  For an affinely foliated {sphere} bundle, the Euler number  satisfies
  
  $$|\mathcal{E}|\leq \frac{||M^n||}{2^n},$$
where $||\cdot||$ is the simplicial volume.

\end{theorem}

\bigskip

We are going to use different terms, so let us start with some definitions.

\medskip

Let $$pr:\mathcal{U}\rightarrow M^n$$ be the universal cover of $M^n$.
A \textit{ fundamental domain} $F \subset \mathcal{U}$ is a closed connected set  with piecewise smooth boundary such that 

\begin{enumerate}
  \item $\mathcal{U}$ is tiled by copies of $F$, that is,
 $$\bigcup_{g \in \Gamma}~ gF=\mathcal{U},$$ where  $\Gamma\cong \pi_1(M^n)$ is the  deck transformations group, and
  \item the sets  $F$ and $gF$ have disjoint interiors whenever $g \neq id$.
\end{enumerate}

\medskip

The \textit{degree  $Deg(y)$ of a point} $y\in \mathcal{U}$ is  the number of elements $g \in \Gamma$ such that $y\in gF$.

A fundamental domain $F$ is called \textit{simple}  if  the dimension of the stratum  $\{y: \ Deg(y) = k\}$  does not exceed $n-k+1$. In particular this means that $\max_{y\in \mathcal{U}}~Deg(y)\leq n+1.$

A point $y$ is called a \textit{vertex} of a simple fundamental domain $F$ if its degree equals  $n+1$.
Its projection $pr(y)$ is called \textit{a vertex of} $M^n$ (related to the fundamental domain $F$). 
Set $\nu=\nu(M^n,F)$ be the number of vertices  of $M^n$.

The boundary $\partial F$ is stratified by the values of $Deg$.
A \textit{facet} of $F$ is a  connected component of the stratum with $Deg=2$. The projection of a facet is called\textit{ a facet of} $M^n$  (related to the fundamental domain $F$). Denote by $\lambda =\lambda(M^n,F)$ the number of facets of $M$. 
{A simple observation is that the number of facets of $F$ is  twice the number of facets of $M^n$.}

The set of facets of $F$ gives rise to a set of generators $\mathcal{G}$ of the group $\Gamma$: two copies of $F$ share a facet iff they differ by a generator.
In other words, 

    $$ g \in \mathcal{G}\   \hbox{ iff } \dim(F\cap gF) =n-1.$$

Clearly, $\mathcal{G}=\mathcal{G}^{-1}$, and $|\mathcal{G}|\leq 2\lambda$.  
\medskip

Our second main result is:

\begin{theorem}\label{ThmMainGeneralCase}
Let~$M^n$ be a closed oriented smooth manifold, and let~$$S^{n-1}\rightarrow E\rightarrow M^n$$ be a flat sphere bundle. If $F$ is a simple fundamental domain, then 
$$|\mathcal{E}|\leq  \frac{\nu}{2^n}\cdot  \left( 1 -\dfrac{n(n+1)}{2}\cdot \dfrac{(2\lambda-1)^{n-1}}{(2\lambda)^{n+1}}\right), $$

where $\nu=\nu(M^n,F)$ is the number of vertices of $M^n$, and $\lambda=\lambda(M^n,F)$ is the number of facets of $M^n$, related to $F$. 
\end{theorem}


\textbf{Example.} For~$n=2$ all flat circle bundles are found by J. Milnor:

\begin{theorem}(Milnor Theorem)\label{MW}\cite{Milnor, MilnorIneq}
An oriented circle bundle~$E\xrightarrow[\text{}]{\pi} S_g$ over a closed oriented surface of genus~$g$ 
is flat if and only if the Euler number of the bundle satisfies 
$$|\mathcal{E}|\leq g-1.$$
\end{theorem}
Consequently,  for~$n=2$ the bound of Theorem \ref{ThmGromovSmillie} is strict, since $||S_g||=4g-4.$
\medskip

\textbf{Remark.} The group~$\pi_1(S_g)$ is amenable if and only {if~$g\leq 1$}, so Theorem~\ref{MW} agrees with Theorem~\ref{ThmMain}.

  For $n=2$, Theorem \ref{ThmGromovSmillie} implies the ``only if'' part of the Milnor theorem.   This shows that one cannot expect a much better estimate of the Euler number.

\medskip

\textbf{Remark.} Another evidence that better estimates of the Euler number may occure only in partial cases is the result of \cite{BucherMonod} which states that  for even $n$, the norm of the universal Euler class is equal to $1/2^n$, and thus the Gromov-Sullivan-Smillie bound is sharp in this case.

\bigskip

\textbf{Example.}
  Let us attach $k$ handles $S^3\times [0,1]$ to the sphere $S^4$. The fundamental group of the resulting manifold is a free group with $k$ generators.
  Cutting each of the handles by $S^3\times {1/2}$, one gets a fundamental domain with $\nu=0$, so by Theorem
  \ref{ThmMainGeneralCase}, there are no affinely foliated bundles over this manifold except for those with $\mathcal{E}=0$.
  
  This result follows also  immediately from {Theorem \ref{ThmGromovSmillie}, } and \cite{Gromov}: the simplicial volume of the sphere with handles vanishes since it is the connected sum of $k$ tori $S^1\times S^3$.

\medskip

\textbf{Example.}
  A simple fundamental domain exists for any $M^n$ and can be constructed in many ways. For example,
  as follows: 
  \begin{enumerate}
                                                                                        \item Equip $M^n$ with a  Riemannian metric.
                                                                                        \item Pick a point $x_0\in M^n$. The associated cut locus $CL(x_0)$
                                                                                        yields a  fundamental domain which is generically simple, \textcolor{red}{} 
                                                                                        \item The cut locus $CL(x_0)$ is stratified according to the number of shortest paths leading from $x\in CL(x_0)$ to $x_0$. Eliminate from $CL(x_0)$ connected strata with number of paths $2$ whenever these two paths give rise to a contractible loop.  The closure of the {preimage} $$pr^{-1}(M^n \setminus CL(x_0))\subset \mathcal{U}$$ gives a simple    fundamental domain.
                                                                                      \end{enumerate}

We shall make use of yet another construction of a simple fundamental domain in Section \ref{SectionFollows}.

\bigskip

\medskip

\subsection{Computation of the Euler number}\label{SecComputeEuler}

The Euler number of an oriented sphere bundle~$E\rightarrow M^n$ is the obstruction to the existence of a continuous section. We briefly recall the following standard construction (see~\cite{FomFu}).

	 Assume
	that a partial section of the bundle~$s: M^n\setminus\{x_i\}\rightarrow E$ is defined everywhere except for a finite set of points~$x_1,...,x_l\in M^n$.
	For each point~$x_i$, choose   a small neighborhood~$U_i$ bounded by a  sphere~$C_{x_i}\subset M^n$. The sphere inherits the orientation from~$M^n$.  
	Next, choose a trivialization of the bundle in the neighborhood~$U_{x_i}$.
	The  restriction of the section~$s$
	defines  a map 
	$$s_{\mid_{S_{x_i}}}:S_{x_i}\rightarrow S^{n-1}.$$  
	The source and the target spaces are two oriented spheres, 
	so the degree of the map is well-defined. Set
	$$\ind_s(x_i):= \deg~ s_{\mid_{S_{x_i}}}.$$ 
	The individual indices~$\ind_s(x_i)$ depend on the section~$s$, but their sum  does not:
	
	\begin{proposition} \cite{Milnor}, \cite{FomFu}.
		In the above notation, the Euler number of the bundle equals the sum of indices:
$$\mathcal{E}(E\rightarrow M^n)=\sum_i \ind_s(x_i).$$
	\end{proposition}\label{PropEulerComput}

This construction applies not only for smooth manifolds, but also for finite cell complexes of dimension~$n$ provided that the points~$x_i$ lie inside~$n$-cells.
\medskip

\subsection{Proof of Theorems~\ref{ThmMain} and \ref{ThmGromovSmillie} via ingredients of M. Gromov, J. Hirsh, W. Thurston, and D. Sullivan}


The proof can be divided into three steps:
\medskip

	\textbf{(1)}	
 D. Sullivan \cite{Sullivan} proved that if an affine sphere bundle has a smooth transverse foliation, then its Euler number does not exceed the number of~$n$-simplices in any triangulation of  the base~$M^n$.

\begin{proof}
  Assume there is an { oriented  bundle  with a transverse affine foliation.}  Take a cell decomposition of~$M^n$ which is dual to some triangulation.
Choose a partial section which is defined on the interiors of the cells by taking parts of leaves of the foliation.  Using the affine structure of the foliation, extend the partial section first to the interiors of the cells of codimension~$1$, then  to the interiors of the cells of codimension~$2$, etc. Eventually one arrives at a continuous partial section~$s$ defined everywhere except for the vertices of the cell decomposition.  It is easy to show that each vertex~$v$
contributes an index~$\ind_s(v)$  whose absolute value is at most~$1$, which completes the proof.
\end{proof}

We are going to make use of this idea in our proof of Theorem~\ref{ThmMain}.
\medskip

\textbf{(2)} The Sullivan's result together with the definition of simplicial norm almost immediately imply that
$$|\mathcal{E}|\leq ||M^n||.$$   Moreover, the Smillie's averaging trick implies that
 $$|\mathcal{E}|\leq {\frac{1}{2^n}||M^n||}.$$

\textbf{(3)}   If~$\pi_1(M^n)$  is amenable, then  the simplicial volume~$||M^n||$  vanishes.
\begin{proof} This follows from the combination of the following facts:
    \textbf{ (A) } If the top  bounded cohomology group is trivial, the simplicial volume~$||M^n||$  vanishes .   \textbf{(B)} Bounded cohomology depends only on the fundamental group \cite{Gromov, Ivanov}.  \textbf{ (C)}  Bounded cohomologies of amenable groups are trivial \cite{HirshThurst}. 
\end{proof}

\section{Random quasisections}

The proof of both theorems ~\ref{ThmMain} and \ref{ThmGromovSmillie} begin by introducing   random quasisections and two related lemmas.

\subsection{Flat quasisections} 

Assume that an affinely foliated {sphere} bundle \newline $\pi: E\rightarrow M^n$ is fixed. Let $pr: \mathcal{U}\rightarrow M^n$ be the universal cover and $Q\subset \mathcal{ U}$ be a  { closed} bounded   subset, such that~$pr(Q)=M^n$.

\medskip


\begin{lemma}
(1) There exists a map
$$q:Q\rightarrow E$$
whose image~$q(Q)$ lies in a single leaf of the foliation and such that the diagram below commutes.

\begin{equation}
\label{V_G_SquareDiagram2}
    \xymatrix{
   \mathcal{U}\ar[drr]_{pr} &&
Q \ar[ll]_{in} \ar[rr]^{q} &&  E \ar[lld]^{\pi} &\\
     && M^n  &\\    }
\end{equation}

(2) The map~$q$ is not unique. For a given point~$a\in Q$ one can set~$q(a)$ to be any \textit{starting point}~$\mathbf{p}$  in the fiber~$\pi^{-1}(pr(a))$. If $Q$ is connected, the choice of a starting point defines~$q$ uniquely.  

Otherwise one chooses a point $a_i$ in each of the connected components. The choice of  $p_i \in \pi^{-1}(a_i)$ defines~$q$ uniquely.\qed
\end{lemma}

The triple~$\mathcal{Q}=(Q,a,\mathbf{p})$  is called a \textit{flat quasisection} related to the foliation (cf.~\cite{PanTurSh}).

The  image $q(Q)$ has no transverse self-crossings, but might have overlappings.  

\medskip

\subsection{The idea of the proof of Theorems \ref{ThmMain} and \ref{ThmMainGeneralCase} in a nutshell}
 Consider a collection of specially designed \textit{parallel} flat  quasisections, that is, a collection of flat quasisection with one and the same $Q$, but with different starting points $\textbf{p}$. Each of the quasisections  comes with its image under central symmetry on the sphere $S^{n-1}$ (this is an application of the Smillie's averaging idea). Based on this collection, we choose a random partial section which is defined everywhere except a finite number of points (this is Kazarian's averaging principle). For the random partial section we compute the expectation of the Euler number. On the one hand, the expectation is equal to $\mathcal{E}$. On the other hand, changing the order of summation gives a local formula for the Euler number.

    \subsection{Collection of parallel flat quasisections}\label{secEssent}
    
 Assume that a simple fundamental domain $F$ is fixed. Its boundary  {is} called \textit{bold},   and  {is} depicted by bold lines in figures.
 
 Let us cut $F$ by piecewise smooth cuts into smaller pieces $d_i$ such that:
 
 \begin{enumerate}
   \item The collection $\{d_i\}$ gives a cell decomposition of $M^n$ which is dual to some triangulation.
   \item The resulting cell complex $\mathcal{D}$  is \textit{regular}, that is, each cell is patched by an injective map.
   
   In simple words, this means that no cell $d_i$ patches to itself.
 \end{enumerate}
 
 The cutting hypersurfaces  {are} called \textit{thin walls} and  {are} depicted by thin lines in figures.
    
     The vertices of the cell complex $\mathcal{D}$ that are the vertices $F$ are called \textit{essential vertices}. The other vertices of $\mathcal{D}$ are called \textit{non-essential}, see Fig. \ref{BoldAndThin}, top.
     
     \medskip
    
    Now let $G\subset \Gamma \cong \pi_1(M^n)$ be some finite set, and let $Q=Q_G$ be the union of fundamental domains:
    $$\mathcal{U} \supset Q=Q_G= \bigcup_{g \in G} gF.$$

    Figure \ref{BoldAndThin}, bottom, shows some $Q_G$ for the   torus. 
    
    For the time being we fix some $G$ and omit the subindex $G$ for brevity.
    
    The cellular structure $\mathcal{D}$ on $F$ gives rise to a cellular structure $\mathcal{D}(Q_G)$ of $Q_G$.
 Denote by  $C$ be the number of all $n$-cells in $\mathcal{D}$.

    Fix a point $a\in Q $, choose $C$ generic points  {(these points will correspond  to colors)} $\mathbf{p}_1,...,\mathbf{p}_C\in \pi^{-1}(a)$ and consider $2C$ flat quasisections
    
    $$\mathcal{Q}_i^+=\mathcal{Q}(Q,a, \mathbf{p}_i)  \hbox{  and } \mathcal{Q}_i^-=\mathcal{Q}(Q,a,- \mathbf{p}_i).$$

    \begin{figure}[h]
 \includegraphics[width=0.5\linewidth]{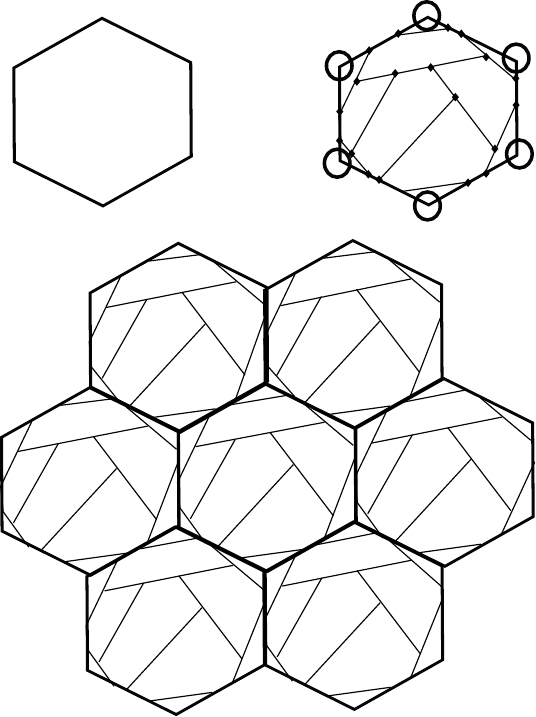}
  \caption {Let $M^n=S^1\times S^1$ be the torus. (Top:)   the fundamental domain $F$ is the hexagon. We cut the hexagon into eleven $2$-cells (so here $C=11$), and depict the   cellular structure.
  Essential vertices are marked by small circles, whereas non-essential vertices are marked by bold dots. In this case $\nu=2$, $\lambda =3$. (Bottom:) In this particular case $Q$ is equal to the union of seven fundamental domains. }
  \label{BoldAndThin}
\end{figure}


\subsection{Randomization}

We are going to construct a random partial section~$s$ defined outside the set of vertices of the cell complex~${\mathcal{D}}$. 
The Euler number is equal to the expectation:
	
	\begin{equation}\label{exp}
	  \mathcal{E}(E\rightarrow M^n)=\mathbb{E}\Big(\sum_i \ind_s(v_i)\Big)= \sum_i \mathbb{E}(\ind_s(v_i)).
	\end{equation}

Here the summation runs over the set of vertices of the complex $\mathcal{D}$.
	Each vertex $v_i$ contributes to $\mathcal{E}$ some \textit{weight} equal to  $\mathbb{E}(\ind_s(v_i)$. The weight will be evaluated.
 
The probability space  {is} a finite one, so taking the expectation  { amounts} to averaging.

 \bigskip

 
   Before we explain the details, let us give a definition.

\begin{definition}
A  \textit{labeled spherical configuration} of~$N$ points is an injective map  from some~$N$-element set to the standard sphere~$S^{n-1}\subset \mathbb{R}^n$.   Sometimes we  record  a configuration as~$(p_1,...,p_N)$ assuming that the $N$-element set is $\{1,...,N\}$, and~$p_i\in S^{n-1}$.

\end{definition}

\textbf{
Step 1.}

Take all the colorings of the $n$-cells of $\mathcal{D}$ into $C$ colors such that different~$n$-cells $d_m$ are colored by different colors (so each of the colors 
is used exactly once). These colorings are called \textit{good}.  Choose a  \textit{random good coloring}, that is, fix a good coloring with equal probability~$1/C!$.

\medskip

\textbf{Step 2}

 Observe that over each of the $n$-cells~$d_m$ the quasisection~$\mathcal{Q}_i^{\pm}$ yields  a finite  covering. For each of the cells $d_m$ choose the $sign$ which is either ``$+$'' or ``$-$'' with probability $1/2$. Next, 
 { choose  a sheet of the covering over ~$d_m$} coming from the quasisection~$\mathcal{Q}^{sign}_i$ whose index $i$ agrees with the color of the cell: {if the cell colored in the $i$-th color, we choose a sheet of $\mathcal{Q}_i^{sign}$.}
 
 We make our random choice
 with equal probability and independently for all the $n$-cells~$d_m$. 
 
\medskip

\textbf{Step 3}

   Next,  we extend each random partial section continuously everywhere except for the set of vertices of $\mathcal{D}$.

 For the extension, we use Sullivan's construction, which we call \textit{the geodesic filling}. Here is a quotation from~\cite{Sullivan}:
    ``For each codimension~$1$ cell choose a point and in the sphere above this point construct an arc between the two points  determined by the previous choices in the two adjacent top cells. Now spread these arcs over the codimension~$1$ cells to partially fill in the discontinuity of our preliminary cross section. Now over a point in a codimension~$2$ cell we find the boundary of a triangle. We fill in this boundary with a triangle as before and
    proceed this way down to the zero dimensional cells... If we inductively choose geodesic arcs, geodesic triangles, geodesic tetrahedra, etc., we would have for each zero cell~$v$ (that is, for each vertex) constructed a map of the boundary of an~$n$-simplex~$\Delta^n$  into~$S^{n-1}$ where each face is carried into 
    a geodesic simplex.''
    This construction works correctly for generic point configurations only. For instance, two antipodal points in a configuration would create a problem since the shortest connecting arc is ill-defined for them.
    
    However, the generic choice of the starting points $\mathbf{p}_1,...,\mathbf{p}_C$ of the quasisections guarantees that in our construction geodesic filling works correctly.
    
    So eventually we arrive at a random  section defined everywhere except for the set of vertices of the complex $\mathcal{D}$.

 \subsection{Evaluation of
~$\mathbb{E}(\ind_s(v))$. Two key lemmas} 
  
  \medskip
  
  A simple fact is:
  \begin{claim}\label{LemmaConv}
     The geodesic filling  associated with a configuration~$(p_1,...,p_{n+1})\subset S^{n-1}\subset \mathbb{R}^n$ induces a map~$\partial \Delta^n \rightarrow S^{n-1}$. The degree of the map  (the index of the configuration) is equal to~$\pm 1$ iff the origin~$O$ belongs to the convex hull~$Conv(p_1,...,p_{n+1})$.  Otherwise the index vanishes. \qed
  \end{claim}
  
  \bigskip

 For a vertex $v$ of~$\mathcal{D}$, take its small neighborhood $U_v$ and consider the set $$\pi^{-1}(U_v) \cap \bigcup \mathcal{Q}_i^\pm .$$ It is a union of connected components which we call \textit{sheets over $v$}. The sheets that project surjectively to $U_v$ are called \textit{regular}. 
 
 For example, the top sheet in Figure \ref{sheets}  is regular, and the other two sheets are not.
 
 Our construction implies  directly:
 
 \begin{claim}
   Let $v$ be a non-essential vertex of $M^n$. Assume that $d_1$ and $d_2$ are two $n$-cells of $\mathcal{D}$ adjacent to $v$ such that $d_1$ and $d_2$ share a $(n-1)$-cell coming from a thin wall.  Then for every sheet $r$ over $v$, 
   the projection $\pi(r)$ intersects the interior of $d_1$ if and only if  $\pi(r)$ intersects the interior of $d_2$, see Fig. \ref{sheets}.\qed
 \end{claim}

\begin{figure}[h]
 \includegraphics[width=0.5\linewidth]{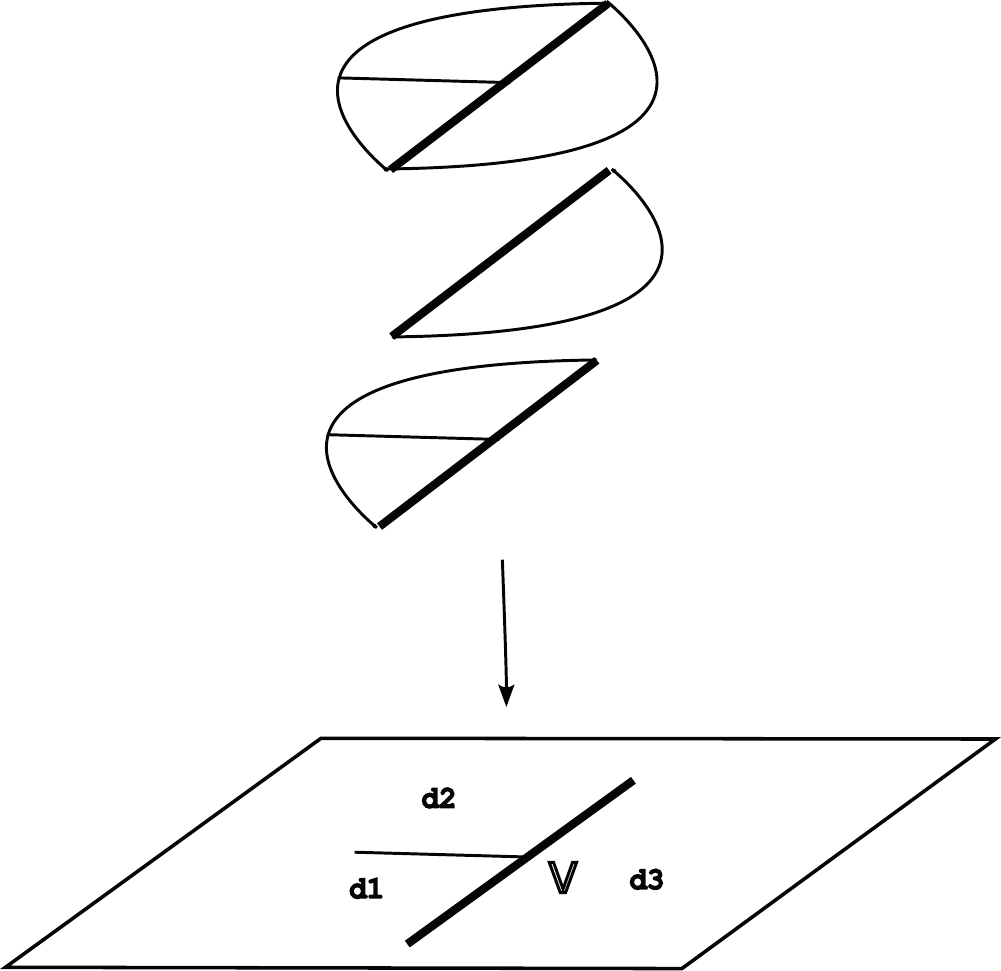}
  \caption {This is the preimage of a neighborhood of a non-essential vertex. Here we have three sheets, the top one is regular.}
  \label{sheets}
\end{figure}

Now come the key lemmas:

  \begin{lemma}\label{LemmaNonEss}
   {The weight $\mathbb{E}(\ind_s(v))$ of a non-essential vertex is zero.}
  \end{lemma}
  
  \begin{proof}
  Let~$v$ be a non-essential vertex of the cell complex~${\mathcal{D}}$.  Consider the $n$-cells of~${\mathcal{D}}$ containing $v$. At least two of them are separated by a thin wall.
  Let us enumerate the cells as $d_1,...,d_{n+1}$ assuming that $d_1$ and $d_2$ share a thin wall, see Fig. \ref{sheets}, and also assuming that the order of the cells agrees with the orientation of $M^n$.
  
  We also assume that the sheets of $\mathcal{Q}_i^\pm$ over the vertex $v$ are also somehow enumerated.

  A random partial section  {in a neighborhood of}  $v$  means a choice of some ordered~$n+1$-point spherical configuration, one point from $\pi^{-1}(v)\cap \mathcal{Q}_i^\pm$  over each of $d_1,...,d_{n+1}$. 
  
  For each $d_i$ adjacent to $v$, we record the choice as $r_k^{l,\pm}$ which means that we choose a sheet number $l$ from the flat quasisection $Q_k^\pm$.
  We assume that the numberings of the sheets $\mathcal{Q}_k^-$ and $\mathcal{Q}_k^+$ are consistent: we have $r_k^{l,+}=-r_k^{l,-}$  for all $k,\ l$.
  Next, we list the choices according to the order  of the cells $d_1,...,d_{n+1}$. 
  
  \medskip
  \textbf{Example:} Notation $(r_2^{3,+}, r_1^{2,-},...)$ means that over the cell $d_1$ we choose the third sheet of $\mathcal{Q}_2^+$, over  the cell $d_2$ we choose
  the second sheet of $\mathcal{Q}_1^-$, etc.

  \medskip
  
  Assume that a configuration $(r_k^*,r_m^*,*)$ arises (with some probability) if $d_1$ is colored in the color $k$, and $d_2$ is colored in the color $m$.
  
 Then the configuration obtained by interchanging the first two entries $(r_m^*,r_k^*,*)$\footnote{Here ''$*$'' denotes something which is identical for the two tuples.} arises (with the same probability) if $d_1$ is colored in the color $m$, and $d_2$ is colored in the color $k$.
 Let us  match $(r_k^*,r_m^*,*)$ with  $(r_m^*,r_k^*,*)$.
  
  \medskip
 \textbf{Example:} $(r_k^{7,+},r_m^{1,-},*)$ is matched with  $(r_m^{1,-},r_k^{7,+},*)$.
  
  \medskip

This is a perfect matching,  and the weights of matched pairs (that is, their contributions to
   $\mathbb{E}(\ind_s(v_i))$)
  sum up to zero. Indeed, we have one and the same point configuration in the sphere. The difference is in the order only. So geodesic fillings contribute indices  with different signs.
  
  \end{proof}

    Let $\mathcal{N}=\mathcal{N}_G$ be the number of $\{g\in G: gF\subset \Int Q_G\}$, and
    
 $\mathcal{N}^\partial=\mathcal{N}_G$ be the number of $\{g\in G: gF \hbox{  intersects } \partial Q_G\}$.
 
 Roughly speaking,  $\mathcal{N}$ is the number of ``interior''  fundamental domains  in $Q_G$, whereas
  $\mathcal{N}^\partial$ is the number of ``boundary'' domains.
 For example, in Fig.\ref{BoldAndThin}, we have  $\mathcal{N}=1,$ and $\mathcal{N}^\partial=6.$

   \begin{lemma}\label{LemmaEss}
   {The absolute value of the weight of an essential vertex is bounded from above: }
     $$|\mathbb{E}(\ind_s(v_i))|\leq \frac{(\mathcal{N}^\partial)^{n+1} +(n+1)\cdot \mathcal{N} \cdot (\mathcal{N}^\partial)^{n}}{2^n \cdot \Big(\mathcal{N}^\partial + \mathcal{N}\Big)^{n+1}}.$$
  \end{lemma}

 \begin{proof}
 Let~$v$ be an essential vertex of~${\mathcal{D}}$.
 Enumerate  the $n$-cells of $\mathcal{D}$ that contain $v$, assuming that the order agrees with the orientation of $M^n$. 
 
 We use here notation from the proof of Lemma \ref{LemmaNonEss}.
 
  A random partial section over  $v$  means a choice of some ordered~$n+1$-point tuple from a spherical point configuration. 
  
  Some of the points correspond to regular sheets  of  $\pi^{-1}(U_v) \cap \mathcal{Q}_i^{j,\pm}$, see Fig. \ref{sheets}.
   Let us denote these sheets by $r_i^{j,\pm}$.
  
  Some of the points correspond to non-regular sheets. Let us denote these sheets by $a_i^{j,\pm}$.
  
  We have to average over all possible choices.
  
  If a configuration contains at least two points coming from regular sheets, its weight can be canceled: indeed, the
  tuples $(*,r_i^{j,\pm},*,r_k^{l,\pm},*)$  and $(*,r_k^{l,\pm},*,r_i^{j,\pm},*)$  contribute  opposite weights. These tuples appear for different colorings, but with equal probabilities. More precisely, we match such two tuples iff $,r_i^{j,\pm}$ and $r_k^{l,\pm}$ are the first two entries of type $r$ that appear in this (ordered) tuple.

  \medskip
 \textbf{ Example:} The tuple $(a_5^{7,+},r_2^{1,-},r_3^{2,+},r_7^{3,+})$ is matched with $(a_5^{7,+},r_3^{2,+},r_2^{1,-},r_7^{3,+})$. These tuples appear for different colorings, but with equal probabilities.
 \medskip
  
  We conclude that the remaining tuples fall into two types: (A) those not containing entries of type $r_i^{j,\pm}$, and (B) those containing exactly one entry of type $r_i^{j,\pm}$.
  
  Next we apply Smillie's averaging. We group the tuples that differ by the upper indices $\pm$ only. Each of the groups contains $2^{n+1}$ tuples, but exactly two of them contribute a non-zero index which is either $1$ or $-1$.

 \medskip
 
 For an essential vertex, the number of all choices of sheets that are not matched, and therefore, are not canceled, is equal to
 
 $$2^{n+1}\cdot C\cdot(C-1)\cdot ...(C-n)\cdot\Big((\mathcal{N}^\partial)^{n+1} +(n+1)\cdot \mathcal{N} \cdot (\mathcal{N}^\partial)^{n}\Big). $$

 Remind that they are grouped into groups of $2^{n+1}$, and each group contributes either $2$ or $-2$.
 
  {So the number of choices with non-zero contribution is }
 
   {$$2\cdot C\cdot(C-1)\cdot ...(C-n)\cdot\Big((\mathcal{N}^\partial)^{n+1} +(n+1)\cdot \mathcal{N} \cdot (\mathcal{N}^\partial)^{n}\Big). $$}
 
 The number of all possible choices is  equal to
 
 $$2^{n+1}\cdot C\cdot(C-1)\cdot ...(C-n)\cdot\Big(\mathcal{N}^\partial + \mathcal{N}\Big)^{n+1} . $$

 {The proof is completed by dividing these two expressions.}

  \end{proof}
  
  \section{Proof of Theorem~\ref{ThmMain}}\label{sec:proof}

\begin{definition}
    A discrete group~$\Gamma$ is called \textit{amenable} if it admits a finitely additive left-invariant probability measure.
\end{definition}

We are going to make use of the \textit{F{\o}lner property} of amenable groups:

\begin{definition}
For a discrete group~$\Gamma$ a F{\o}lner sequence is a sequence~$\{\Phi_i\}$ of nonempty finite subsets of~$\Gamma$ such that
$$\dfrac{|g\Phi_i\Delta \Phi_i|}{|\Phi_i|}\rightarrow 0$$ f
or every~$g\in \Gamma$, where $\Delta$ denotes the symmetric difference.
\end{definition}

\begin{lemma} \cite{Folner}
    A group has a F{\o}lner sequence if and only if it is amenable.
\end{lemma}

Assume that~$\Gamma \cong \pi_1(M^n)$ is amenable, and let~$\{\Phi_i\}$ be a F{\o}lner sequence. 

Let~$\varepsilon >0$ be a small number and $B_1$ be a ball in $Cayley(\mathcal{G})$ of radius $1$.
Since~$\{\Phi_i\}$ is a F{\o}lner sequence,  for every~$g\in \Gamma$ there exists~$I \in \mathbb{N}$ such that~$\dfrac{|g\Phi_I\Delta \Phi_I|}{|\Phi_I|}<\varepsilon$. Since~$|B_1|<\infty$,  such number~$I$ can be chosen uniformly for all~$g\in B_1$. Therefore, for this~$I$ we have

$$\dfrac{\left|\left(\bigcup\limits_{g\in B_1}g\Phi_I \right)\setminus \Phi_I\right|}{\left|\Phi_I\right|} = \dfrac{\left|\left(\bigcup\limits_{g\in B_1}g\Phi_I \right)\Delta \Phi_I\right|}{\left|\Phi_I\right|}<\varepsilon \cdot |B_1|.$$

\medskip

Now consider the set $G = \bigcup\limits_{g\in B_1}g\Phi_I \subset \Gamma$ and the corresponding quasisection $Q = \bigcup\limits_{g\in G}gF$.

Let $\mathcal{N}$ and $\mathcal{N}^{\partial}$ be as before.
In terms of the F{\o}lner sequence:
$$\mathcal{N} \geq \left|\Phi_I\right|, \ \ \  \mathcal{N}^\partial\leqslant \left|\left(\bigcup_{g\in B_1}g\Phi_I \right)\setminus \Phi_I\right|.$$

Indeed, by Lemma \ref{LemmaInterior} the domains corresponding to $g\in \Phi_I$ are inner.

Informally, F{\o}lner sets of an amenable group are the sets whose interior is much larger than the boundary. Therefore: 

 \begin{corollary}\label{limit0} 
We have $\mathcal{N}^{\partial}/\mathcal{N}<\varepsilon \cdot |B_1|$. 
 \end{corollary}

 \begin{proof}
     Follows from the above inequalities and the F{\o}lner property.
 \end{proof}

Now we are ready to prove the theorem. By Lemma \ref{LemmaEss}, 

$$|\mathcal{E}| \leq \nu(F)\cdot \frac{(\mathcal{N}^\partial)^{n+1} +(n+1)\cdot \mathcal{N} \cdot (\mathcal{N}^\partial)^{n}}{2^n \cdot \Big(\mathcal{N}^\partial + \mathcal{N}\Big)^{n+1}}.$$

Denoting $\mathcal{N}^{\partial}/\mathcal{N}$ by $\mathbf{x}$ and using Corollary \ref{limit0},

$$|\mathcal{E}| \leq \nu(F) \cdot  \dfrac{\mathbf{x}^{n+1} + (n+1)\mathbf{x}^n}{(1+\mathbf{x})^{n+1}}\leq Const \cdot \varepsilon^n.$$ 
 
Since we can choose~$\varepsilon$ to be arbitrarily small,~$|\mathcal{E}|$  vanishes.

 \section{Proof of Theorem \ref{ThmMainGeneralCase}}

 Due to Lemma \ref{LemmaEss}, it is sufficient to prove that with a smart choice of $G$, each essential vertex contributes at most $\frac{1}{2^n}\cdot\left( 1-\frac{n(n+1)}{2}\cdot \frac{(2\lambda-1)^{n-1}}{(2\lambda)^{n+1}}\right)$.

    Now let $Cayley(\mathcal{G})$  be the Cayley graph of the group $\Gamma$ associated to the set of generators $\mathcal{G}$. Set $G=B_R$ be the ball of radius  $R$ in $Cayley(\mathcal{G})$, and set $$\mathcal{U} \supset Q=Q_R= \bigcup_{g \in B_R} gF.$$ 
    
As before, let $\mathcal{N}=\mathcal{N}(R)$ be the number of $\{g\in B_R: gF\subset \Int Q_R\}$, and
 
 $\mathcal{N}^\partial=\mathcal{N}(R)$ be the number of $\{g\in B_R: gF \hbox{  intersects } \partial Q_R\}$.
 
 So, $\mathcal{N}^\partial +\mathcal{N}$ is equal to $|B_R|$, or, equivalently, to the number of copies of $F$ contained in $Q_R$.

 \medskip

    Let us estimate the expression from Lemma \ref{LemmaEss} in terms of $\lambda$ and $n$. Set $\mathbf{x} = \dfrac{\mathcal{N}^\partial}{\mathcal{N}}$. Then $$ \mathbb{E}(\ind_s(v_i)) \leq \dfrac{(\mathcal{N}^\partial)^{n+1} +(n+1)\cdot \mathcal{N} \cdot (\mathcal{N}^\partial)^{n}}{\left(\mathcal{N}^\partial + \mathcal{N}\right)^{n+1}} = \dfrac{\mathbf{x}^{n+1} + (n+1)\mathbf{x}^n}{(1+\mathbf{x})^{n+1}}. $$ 

It remains to estimate $\mathbf{x}$.

\begin{lemma}\label{LemmaInterior} Let $F$ be a simple fundamental domain. Then $$F\subset \Int \bigcup_{g \in B_1} gF,$$ where $B_1$ is the ball of radius~$1$ in the  graph $Cayley(\mathcal{G})$.

\end{lemma}

\begin{proof} Simplicity of $F$ implies that for every $v\in F$ all the fundamental domains, containing $v$ have a common codimension-one face with $F$. Hence, $v \in \Int \bigcup_{g \in B_1} gF,$ by the definition of $\mathcal{G}$.
\end{proof}

     Lemma \ref{LemmaInterior} implies that  $gF$ is contained  in the interior of $Q_R$ for all $g\in B_{R-1}$. So $\mathcal{N}\geq |B_{R-1}|$, \ \  $\mathcal{N}^{\partial}\leq |B_R| - |B_{R-1}|$ and $$\mathbf{x} = \dfrac{\mathcal{N}^\partial}{\mathcal{N}}\leq \dfrac{|B_R| - |B_{R-1}|}{|B_{R-1}|}.$$ 
    
    \begin{lemma}\label{lemmaNonAm} For $R>2$ it holds that in $Cayley(\mathcal{G})$:
    $$\dfrac{|B_R| - |B_{R-1}|}{|B_{R-1}|} \leq 2\lambda-1.$$ 
    
    \end{lemma}

    \begin{proof}
    Consider  arbitrary $g\in B_{R-1}\setminus B_{R-2}$. There is at least one edge $[g  g_0] \in Cayley(\mathcal{G})$ such that $g_0 \in B_{R-2}$ --- it is an ``ancestor'' of $g$. 

    Hence, at most $|\mathcal{G}| -1$ edges of $Cayley(\mathcal{G})$ connect $g$ with vertices in $B_R\setminus B_{R-1}$. 
    
    So, $$\dfrac{|B_R| - |B_{R-1}|}{|B_{R-1}|}\leq |\mathcal{G}|-1 \leq 2\lambda-1.$$
    
    \end{proof}
    Finally, by Lemma \ref{lemmaNonAm}, $\mathbf{x}\leq 2\lambda -1$. It remains to note that the function $f(\mathbf{x}) = \dfrac{\mathbf{x}^{n+1} + (n+1)\mathbf{x}^n}{(1+\mathbf{x})^{n+1}}$ increases for positive $\textbf{x}$, therefore $$\mathbb{E}(\ind_s(v_i)) \leq  \dfrac{\mathbf{x}^{n+1} + (n+1)\mathbf{x}^n}{(1+\mathbf{x})^{n+1}}\leq $$$$\dfrac{(2\lambda-1)^{n+1} + (n+1)(2\lambda-1)^n}{(2\lambda)^{n+1}}\leq 1 -\dfrac{n(n+1)}{2}\cdot \dfrac{(2\lambda-1)^{n-1}}{(2\lambda)^{n+1}}.$$

      \section{{Theorem \ref{ThmGromovSmillie} follows from  Theorem \ref{ThmMainGeneralCase}. }}\label{SectionFollows}

Remind first that
an $n$-dimensional (closed, compact)\textit{ pseudomanifold} is a quotient space obtained from a finite collection of standard $n$-simplices  $\Delta_i^n$ by identifying their $(n-1)$- faces in pairs via affine isomorphisms, with each $(n-1)$- face identified with exactly one other such face.

    A pseudomanifold  is called \textit{orientable} if  orientations $\varepsilon_i$ on the simplices   may be
chosen in such a way that the affine identifications between the paired faces
(endowed with the induced orientations) are all orientation-reversing. This condition  means that the homological sum of the simplices is a cycle.

\medskip 

Given an oriented locally trivial $S^{n-1}$-bundle over an oriented pseudomanifold $X$, its \textit{ Euler number $\mathcal{E}$} is defined as the value of its Euler class on  the cycle
        $[X]:=[\sum_i \varepsilon_i\Delta^n_i]\in H_n(X)$.
        
        Computation of the Euler number described in Section \ref{SecComputeEuler} using partial sections extends to $S^{n-1}$-bundles over $X$ provided that the  points od discontinuity $\{x_i\}$ lie in $X\setminus \Skel_{n-2}X$.  The key property that makes this construction work (and that will be used throughout what follows) is that every $S^{n-1}$-bundle has a continuous section defined on the $(n-1)$-skeleton of $X$.

     \begin{proposition}\label{PropPseudomanifold}
     
     Theorem \ref{ThmMainGeneralCase} remains valid for an arbitrary compact oriented pseudomanifold $X$  with $\nu$ equal to the number of $n$-simplices, and $\lambda$ equal to the number of edges of $X$.
   
   \end{proposition}
   \begin{proof}
     Let us cut $X$ into elementary \textit{bricks}. Take the barycentric subdivision $\sd X$ and for each vertex $V$ of $X$ associate $Brick(V) = \Star_{\sd X}V$, which is the union of simplices of the barycentric subdivision that are incident to $V$. 
     
    In the universal cover of $X$ select a fundamental domain by choosing a connected collection of bricks. The vertices of the fundamental domain correspond to $n$-simplices of $X$, and the facets correspond to edges of $X$.
     
     Since the cell complex $X$ is not necessarily regular, our bricks need not be simply connected. However, they can be subdivided, if necessary, by proceeding as in Section \ref{secEssent} and introducing non-essential vertices.

     Now it remains to repeat literally the proof of Theorem \ref{ThmMainGeneralCase}: we first construct a random partial section over the interiors of the (subdivided) bricks. We then extend it by geodesic filling across the interiors of pairwise intersections of bricks, followed by triple intersections, and so on.
   \end{proof}

Now we are ready to prove that Theorem \ref{ThmGromovSmillie} follows from  Theorem \ref{ThmMainGeneralCase}.
Keeping in mind that the \textit{rational simplicial volume}  $||M||_{\mathbb{Q}}$  is equal to $||M|| $,
consider  a rational singular cycle $\overline{Z}=\sum a_j \overline{\sigma_j}$  representing the fundamental class $[M^n]$. The cycle $\overline{Z}$ is a multiple of  an integer cycle, that is,  $Z=d\cdot \overline{Z}$ such that $Z =\sum \limits_{i=1}^{N} \varepsilon_i\sigma_i$, with $\varepsilon_i = \pm 1$. Clearly,  $||\overline{Z}|| = N/d$. 

It is known  (\cite{Hatcher}) that $Z$ (as any other integer cycle)  can be represented geometrically by a suitable $\Delta$-\textit{complex} $X$ (in terminology of \cite{Hatcher}) which is  an orientable {pseudomanifold}  patched of $N$ $n$-dimensional simplices, and a map $f:X \rightarrow M$ such that $f_*(X)=d\cdot [M]$.

   The pullback of $E\rightarrow M$ gives a transversally foliated fiber bundle  $\widetilde{E}\rightarrow X$. 

\[
\begin{CD}
\widetilde{E}=f^*E @>>> E \\
@VVV         @VVV \\
X @>f>> M
\end{CD}
\]

 Since $f_*[X] = [Z] = d\cdot [M]$ for $f_*:H_n(X)\rightarrow H_n(M)$, we have $$\mathcal{E}(\widetilde{E}) = d\cdot \mathcal{E}(E).$$

  We now show that Theorem \ref{ThmGromovSmillie} follows from our Theorem \ref{ThmMainGeneralCase}.   An interesting question is if there are examples that are covered by our theorem but are not covered by the former. Apply Proposition \ref{PropPseudomanifold} to the pullback bundle $\widetilde{E}$.

     $$\mathcal{E}(E) = \frac{\mathcal{E}(\widetilde{E})}{d}\leq  \frac{\nu(X)}{d\cdot 2^n} =  \frac{\hbox{ number of $n$-simplices in $X$}}{d\cdot 2^n}=  \frac{N}{d\cdot 2^n}=\frac{||\overline{Z}||}{ 2^n}.$$

 Taking the infimum over all such $\overline{Z}$, we get $\|M^n\|/2^n$ on the right hand side.

\end{document}